\documentclass{amsart}
\usepackage{amsthm}
\usepackage{amssymb}
\usepackage[utf8]{inputenc}
\usepackage{tikz}
\usetikzlibrary{decorations.pathreplacing,decorations.markings}
\usepackage{games}
\newtheorem{theorem}{Theorem}
\newtheorem*{theorem*}{Theorem}

\theoremstyle{definition}
\newtheorem{question}[theorem]{Question}

\newtheorem{lemma}[theorem]{Lemma}
\newtheorem{claim}[theorem]{Claim}
\theoremstyle{definition}
\newtheorem{definition}[theorem]{Definition}

\newcommand{\Z}{\mathbb{Z}}

\newcommand{\R}{\mathbb{R}}

\newcommand{\bGamma}{\boldsymbol{\Gamma}}
\newcommand{\bG}{\boldsymbol{\Gamma}}

\newcommand{\bP}{\boldsymbol{\Pi}}
\usepackage{amssymb}

\newcommand{\bD}{\boldsymbol{\Delta}}

\newcommand{\bS}{\boldsymbol{\Sigma}}

\newcommand{\I}{{\boldsymbol{I}}}
\newcommand{\II}{{\boldsymbol{I}\kern-0.05cm\boldsymbol{I}}}

\newcommand{\res}{\mathord \upharpoonright}

\newcommand{\abs}[1]{\left \lvert #1 \right \rvert}

\newcommand{\ignore}[1]{}
\newcommand{\modsim}{/\kern-0.1cm\sim}

\DeclareMathOperator{\dom}{dom}
\newcommand{\ww}{{\omega^{\omega}}}
\newcommand{\sm}{\setminus}
\newcommand{\conc}{{}^\smallfrown}
\newcommand{\tA}{\tilde{A}}

\newcommand{\gegd}{(\bG,E_G)\text{-determinacy}}

\newcommand{\ged}{(\bG,E)\text{-determinacy}}
\newcommand{\getd}{(\bG,E,T)\text{-determinacy}}

\newcommand{\setof}[1]{\left\{#1\right\}}
\newcommand{\suchthat}{\colon~}

\tikzset{
  on each segment/.style={
    decorate,
    decoration={
      show path construction,
      moveto code={},
      lineto code={
        \path [#1]
        (\tikzinputsegmentfirst) -- (\tikzinputsegmentlast);
      },
      curveto code={
        \path [#1] (\tikzinputsegmentfirst)
        .. controls
        (\tikzinputsegmentsupporta) and (\tikzinputsegmentsupportb)
        ..
        (\tikzinputsegmentlast);
      },
      closepath code={
        \path [#1]
        (\tikzinputsegmentfirst) -- (\tikzinputsegmentlast);
      },
    },
  },
  mid arrow/.style={postaction={decorate,decoration={
        markings,
        mark=at position .5 with {\arrow[#1]{stealth}}
      }}},
}

\title{Equivalence Relations and Determinacy}
\author{Logan Crone}
\address{Logan Crone, University of North Texas, Department of Mathematics, 1155 Union Circle
  \#311430, Denton, TX 76203-5017, USA}
\email{logancrone@my.unt.edu}
\author{Lior Fishman}
\address{Lior Fishman, University of North Texas, Department of Mathematics, 1155 Union Circle \#311430, Denton, TX 76203-5017, USA}
\email{lior.fishman@unt.edu}
\author{Stephen Jackson}
\address{Stephen Jackson, University of North Texas, Department of Mathematics, 1155 Union Circle \#311430, Denton, TX 76203-5017, USA}
\email{stephen.jackson@unt.edu}

\begin{document}
\maketitle

\begin{abstract}
We introduce the notion of $(\bG,E)$-determinacy for $\bG$ a pointclass and $E$
an equivalence relation on a Polish space $X$. A case of particular interest is the case
when $E=E_G$ is the (left) shift-action of $G$ on $S^G$ where $S=2=\{0,1\}$ or $S=\omega$.
We show that for all shift actions by countable groups $G$, and any ``reasonable''
pointclass $\bG$, that $\gegd$ implies $\bG$-determinacy. We also prove a corresponding result
when $E$ is a subshift of finite type of the shift map on $2^\Z$.
\end{abstract}

\section{Introduction}
For $X=2^\omega$ or $X=\ww$,  $E$ any equivalence relation on $X$,
and $\bG$ any pointclass (a collection of subsets of Polish spaces
closed under continuous preimages, the reader can consult
\cite{MoschovakisBook} and \cite{KechrisBook}
for background on the basic notions of  descriptive set theory
which we use throughout), there is a natural notion of $\ged$.
Namely, this asserts that any $A \subseteq X$ in $\bG$ which is $E$-invariant is determined.
Similarly, if $G$ is a countable group, and we fix an enumeration of $G$, 
then there is a natural notion of $\ged$ for sets $A\subseteq 2^G$ or $A\subseteq \omega^G$
in $\bG$ (under the natural identification of $2^G$ with $2^\omega$ via the enumeration
of $G$). We will give a more general definition of $\ged$ for arbitrary
Polish spaces $X$ and equivalence relations $E$ on $X$ in \S\ref{sec:ged}.
However, even the special cases just mentioned have risen in various contexts.
For example, when $E$ is the Turing equivalence relation on $2^\omega$, then the
question of when $\bG$ Turing-determinacy implies full $\bG$-determinacy has
been an important question in modern logic. Harrington \cite{Harrington1978}
showed that $\bS^1_1$ Turing-determinacy is equivalent to $\bS^1_1$-determinacy.
Woodin showed that in $L(\R)$ Turing determinacy implies full determinacy.
It is open in general for which pointclasses $\bG$ we have that
$\bG$ Turing-determinacy implies $\bG$-determinacy.

In another direction, in recent years arguments involving Borel
determinacy have had fruitful applications to the theory of Borel
equivalence relations.
The determinacy of Borel games is a fundamental result of Martin \cite{Martin1975},
\cite{Martin1985}. Despite the central significance of this result
in modern logic, this result has until recently found relatively few applications
as a proof technique. 
Recently, however, Marks~\cite{Marks2016} uses Borel
determinacy arguments to get lower-bounds on the Borel chromatic number $\chi_B$ for
free actions of free products of groups, in particular, the lower-bound
that $\chi_B(2^{F_n})\geq 2n+1$ for the chromatic number for the free part
of the shift action (defined below) of the free group $F_n$ on the
space $2^{F_n}$.
See also \cite{KechrisMarks} for a detailed account of recent advances
in the theory of descriptive graph combinatorics including the use of
Borel determinacy. 
There are currently no other proofs of this result, and so
the introduction of determinacy methods into the subject represents an important
connection.

In this paper we begin to investigate the general question of when
$\ged$ implies $\bG$-determinacy. Again, we formulate  the notion of
$\ged$ more generally for arbitrary Polish spaces and equivalence relations
in \S\ref{sec:ged}. First, however, we investigate the special case
where $X=2^G$ (or $X=\omega^G$) and $E$ is the equivalence relation induced by the shift action
of $G$ on $X$.

One of our main results is that for any countable group $G$
and any pointclass $\bG$ satisfying some reasonable closure properties (a ``reasonable pointclass'')
that $\gegd$ implies $\bG$-determinacy. The proof passes through a property of $G$
which we call {\em weak amenability} which may be of interest elsewhere. 
In \S\ref{sec:shift} and \S\ref{sec:wa} we introduce the notion of weak amenability and
establish some basic properties. In \S\ref{sec:shiftproof} we use weak amenabilty to give
a combinatorial proof of the determinacy result.

In \S\ref{sec:ged} we give the general definition of $\ged$ and then in \S\ref{sec:ss} we prove our
second main result that for equivalence relations $E$ induced by
subshifts of $2^\Z$ of finite type and reasonable pointclasses
$\bG$ that $\ged$ implies $\bG$-determinacy. The proof follows the outline of the
first main theorem, but has extra complications involving the combinatorics of the
subshift.

\section{Shift Actions} \label{sec:shift}

Let $G$ be a countable group. The (left) shift action of $G$ on $S^G$ is the action defined by
$g \cdot x(h)=x(g^{-1}h)$. The cases of primary interest are when $S=2=\{ 0,1\}$ and
$S=\omega$. In either of these cases, we refer to $S^G$ as the shift space, and note that
the action of $G$ on $S^G$ is continuous (with the usual product of the discrete
topologies on $S^G$). Let $E_G$ denote the equivalence relation on $S^G$ induced by
the shift action.

Let  $\pi \colon \omega \to G$ be a bijection, which we view as an enumeration of the group $G$,
$G=\{ \pi(0),\pi(1),\dots\}$. We also write $g_i$ for $\pi(i)$ to denote the $i$th group element.
The enumeration $\pi$ induces a homeomorphism, which we also call $\pi$,  between $S^\omega$ and $S^G$, namely
$\pi(x)=y$ where $y(g_n)=x(n)$.

\begin{definition}
Let $\bG$ be a pointclass. Let $G$ be a countable group and $S=2$ or $S=\omega$, and let $E_G$
be the shift equivalence relation on $S^G$. We say $\gegd$ holds if for all $A\subseteq S^G$
which are in $\bG$ and $E_G$-invariant we have that $\pi^{-1}(A)\subseteq S^\omega$
is determined, for all enumerations $\pi$ of $G$. 
\end{definition}

Our main theorem will require a mild closure hypothesis on the pointclass $\bG$
which we now state.

\begin{definition}
We say a pointclass $\bG$ is {\em reasonable} if
\begin{enumerate}
\item
$\bG$ is closed under unions and intersections with $\bD^0_6$ sets.
\item
$\bG$ is closed under substitutions by $\bD^0_6$-measurable functions.
\end{enumerate}
\end{definition}

We note that all levels of the Borel hierarchy past the finite levels are reasonable, as are
all levels of the projective hierarchy.

The next result is our main result connecting $\gegd$ with full $\bG$-determinacy.

\begin{theorem} \label{st}
For any countable group $G$ and any reasonable pointclass $\bG$,
$\gegd$  implies $\bG$-determinacy.
\end{theorem}

The proof of Theorem~\ref{st} will involve a weak form of amenabilty of groups which
we simply call {\em weak amenability}. We give this definition next.

\begin{definition} \label{def:wa}
Let $G$ be a countable group.  We say $G$ is \emph{weakly amenable} if either $G$ is finite,
or if there is an equivalence relation $\sim$ on $G$ such that
\begin{enumerate}
\item \label{app1}$G\modsim$ is infinite
\item \label{app2} $\forall g \in G~\exists b(g) \in \mathbb{N}~\forall C \in G\modsim
\abs{\setof{C' \in G\modsim \suchthat gC \cap C' \neq \emptyset}} \leq b(g)$
\end{enumerate}and an increasing sequence of finite sets $A_n\subseteq G \modsim$,
such that $G\modsim\ =\bigcup_{n \in \omega} A_n$ such that for any $g \in G$
\begin{equation} \label{fc}
\lim_{n \to \infty}\frac{\abs{\{ C \in A_n \colon gC \subseteq \cup A_n \} }}
{\abs{A_n }}=1
\end{equation}
\end{definition}

We call an  equivalence relation $\sim $ on a group $G$ which satisfies both conditions (\ref{app1}) and (\ref{app2})
of Definition~\ref{def:wa} {\em appropriate}. We note that the equivalence classes $C$
in an appropriate equivalence relation need not be finite themselves, but in the definition of
weak amenability, the sets $A_n$ are finite (i.e., they are finite sets of equivalence classes).

We note that every amenable group $G$ is weakly amenable. This follows taking $\sim$
to be the equality equivalence relation on $G$. Note that the equality equivalence
relation  on $G$ is an appropriate equivalence relation (with $b(g)=1$ for every $g\in G$).

We will prove Theorem~\ref{st} by showing two separate results, one of which is a purely
algebraic result concerning weak amenability, and the other a pure game argument. 
The algebraic result is the following:

\begin{theorem} \label{thm:was}
Every infinite group  has an infinite weakly amenable subgroup.
\end{theorem}

The game argument is given in the following theorem.

\begin{theorem} \label{thm:ga}
If $G$ is a countable group which has an infinite weakly amenable subgroup,
then for every reasonable pointclass $\bG$ we have that
$\gegd$ implies $\bG$-determinacy.
\end{theorem}

\section{Weak Amenability} \label{sec:wa}

In this section we establish that certain classes of groups are weakly amenable,
including all the amenable groups and free groups, and prove Theorem~\ref{thm:was}.

The following two lemmas are easy and well-known. 

\begin{lemma} \label{lem:nt}
If $G$ is a non-torsion group, then $G$ has an infinite weakly amenable subgroup.
\end{lemma}

\begin{proof} 
This is immediate as an element of infinite order generates an infinite cyclic
subgroup, which is amenable and so weakly amenable.
\end{proof}

\begin{lemma} \label{lem:lf}
If $G$ is locally finite then $G$ is amenable, and so weakly amenable.
\end{lemma}

\begin{proof}
We may write $G=\bigcup G_n$, an increasing union of finite subgroups. The $G_n$
can be used as Folner sets to witness the amenability of $G$.
\end{proof}

\begin{proof}[Proof of Theorem~\ref{thm:was}]
We may assume without loss of generality that $G$ is an infinite countable group.
By Lemma~\ref{lem:nt} we may assume that $G$ is a torsion group. By Lemma~\ref{lem:lf}
we may assume that $G$ is not locally finite. Then $G$ contains an infinite
subgroup $H=\langle F, g\rangle$ generated by a finite subgroup $F\leq G$ and
an element $g \in G$ of finite order. If suffices to show that any such group $H$
is weakly amenable.

Every element $h\in H$ can be written (not uniquely) in the form
$h=f_1 g^{a_1} f_2 g^{a_2}\cdots f_n g^{a_n}$ where $f_i\in F$ and
$a_i$ are positive integers less than the order of $g$.
We call $n$ the length of this representation of $h$. We let
$|h|$ denote the minimum length of a representation of $h$.
Note that $|h^{-1}| \leq |h|+1$ for any $h \in H$. 
We easily have that $| h_1 h_2|\leq |h_1|+ |h_2|$ and also
$|h_1h_2| \geq | |h_1|-|h_2||-1$.

We let $\sim$ be the equivalence relation on $H$ given by $h_1 \sim h_2$ iff
$|h_1|=|h_2|$. Each equivalence class is finite as $F$ is finite
and $g$ has finite order. So, $H \modsim$ is infinite.
Let $h \in H$. By the above observations we have that for any
$k \in H$ that $ |k|-|h|-1 \leq |hk|\leq |k|+|h|$ and we may
take $b(h)= 2|h|+2$ to satisfy (\ref{app2}) of Definition~\ref{def:wa}.
Thus $\sim$ is an appropriate equivalence relation on $H$.

Let $A_n=\{ [h]_\sim \colon |h|\leq n\}$, so $|A_n|=n+1$. For $h \in H$ we have that
$\{C\in A_n \colon h C\subseteq \cup A_n\} \supseteq
\{ [k] \in A_n\colon |k|\leq  n-|h| \}$, and so
$|\{C\in A_n \colon h C\subseteq \cup A_n\}| \geq n-|h|+1$, and
Equation~\ref{fc} follows.
\end{proof}

The argument above for the proof of Theorem~\ref{thm:was} in fact shows
the following.

\begin{theorem}
Every finitely generated group is weakly amenable.
\end{theorem}

\begin{proof}[sketch of proof] Assume $G$ is infinite and finitely generated.
Let $S=\{ g_1,\dots, g_n\}$ be a finite generating set for
$G$. For $g \in G$, let $|g|$ be the minimal length of a word representing $g$ using  the symbols
$g_i$, $g^{-1}_i$, for $g_i \in S$ (we use only $g_i$ and $g^{-1}_i$ here, not other powers of
the $g_i$). We define $\sim$ in the same way as before (i.e., $g\sim h$ iff $|g|=|h|$).
Note that $G\modsim$ is still infinite with this modification. 
The rest of the argument proceeds as before.
\end{proof}

In fact, as pointed out to us by Simon Thomas, this argument shows the following even more general fact.

\begin{theorem}
If $G$ is a countable group which has a Cayley graph with an infinite diameter,
then $G$ is weakly amenable.
\end{theorem}
$\square$

In particular, the free group $F_\omega$ on infinitely many generators is weakly amenable.

\begin{question}
Is every countable group weakly amenable?
\end{question}

\section{Proof of Theorem~\ref{st}} \label{sec:shiftproof}

In this section we use a game argument to prove Theorem~\ref{thm:ga},
which in view of Theorem~\ref{thm:was} implies Theorem~\ref{st}.
For convenience we restate Theorem~\ref{thm:ga}:

\begin{theorem*}
If $G$ is a countable group which has an infinite weakly amenable
subgroup, then for every reasonable pointclass $\bG$ we have that
$\gegd$ implies $\bG$-determinacy
\end{theorem*}

\begin{proof}
Let $H \leq G$ be an infinite weakly amenable subgroup.  Let the
equivalence relation $\sim$ on $H$ and the the sequence of sets
$\setof{A_n}_{n \in \omega}$ witness this.
We recall here that the sets $A_n$ are finite subsets of the quotient
$H\modsim$. Without loss of
generality, assume that the sequence $\setof{A_n}$ is such that
$\lim_{n \to \infty} \frac{|A_{n+1} \setminus A_n|}{|A_{n+1}|} = 1$.

Define $H_1$ and $H_2$ subsets of $H$ by
\[H_1 = \cup\left(A_0 \cup \bigcup_{n \in \omega} A_{2n+2} \setminus A_{2n+1} \right)\]
\[H_2 = \cup\left(\bigcup_{n \in \omega} A_{2n+1} \setminus A_{2n}\right) \]
Let $\setof{g_k H \suchthat k \in \omega}$ enumerate the cosets of $H$
in $G$.  Let $G_\I = \bigcup_{k \in \omega} g_k H_1$ and $G_\II =
\bigcup_{k \in \omega} g_k H_2$. Clearly $G_\I \cap G_\II=\emptyset$ and $G=G_\I \cup G_\II$.

Let $A \subseteq X^\omega$ be in $\bGamma$.  We'll define an alternate
payoff set $\tilde A \subseteq X^G$ which is $E_G$-invariant and
simulate the game $A$ by the game $\tilde A$ in which player $\I$ makes moves
corresponding to $g \in G_\I$ and player $\II$ makes moves
corresponding to $g \in G_\II$ (we assume the enumeration $\pi$
satisfies $\pi(g)$ is even for $g \in G_\I$, and $\pi(g)$
is odd for $g \in G_\II$).

We will define sets of rules, which if followed by both players
will enforce that each player in the game $\tilde A$ eventually specifies
moves in $A$ to play.
First, we partition the positive even integers into infinitely many
disjoint subsequences $\setof{\setof{c_{n, j}}_{j \in \omega}}_{n \in\omega}$,
and the odd integers  also into $\setof{\setof{d_{n,j}}_{j \in \omega}}_{n \in \omega}$.
Let $B^\I_{n, j}$ denote
$A_{c_{n, j}} \setminus A_{c_{n, j}-1}$ and $B^\II_{n, j}$ denote
$A_{d_{n, j}} \setminus A_{d_{n, j}-1}$. Note that $\lim_{j \to \infty}
\frac{ |B^\I_{n,j}|}{|A_{c_{n,j}}|}=1$, and likewise for $B^\II_{n,j}$.

We will have the players
specify in $\tilde{A}$ their $n$th move in $A$ by playing more
of those moves (by proportion) on the rounds corresponding to
$B^\I_{n, j}$ (or $B^\II_{n, j}$ respectively).  In order to
successfully specify a move, they must have that the limit as $j \to
\infty$ of the proportion of \emph{classes} $C$ for which \emph{all}
the moves in $C$ are the same goes to $1$.

Now we give the formal definition of the rules which will enforce the
correct encoding of moves from $A$ into $\tilde A$.
\[R_n^\I = \setof{x \in X^G \suchthat \exists m \forall k
\lim_{j \to \infty}\frac{\abs{\setof{C \in B^\I_{n, j} \suchthat \forall h
\in C~ x(g_kh)= m}}}{\abs{B^\I_{n, j}}}=1}\]
\[R_n^\II = \setof{x \in X^G \suchthat \exists m \forall k \lim_{j \to \infty}
\frac{\abs{\setof{C \in B^\II_{n, j} \suchthat \forall h \in C~ x(g_kh)= m}}}{\abs{B^\II_{n, j}}}=1}\]

We claim these rules are invariant.  Suppose $x \in R_n^\I$ and $g \in G$.
Let $m$ witness the fact that $x \in R_n^\I$, and fix $k \in \omega$.
We want to show that $m$ also witnesses  $g\cdot x \in R_n^\I$, or in other
words that
\[\lim_{j \to \infty}\frac{\abs{\setof{C \in B^\I_{n, j}\suchthat
\forall h \in C~ g \cdot x(g_kh)= m}}}{\abs{B^\I_{n, j}}}=1\]

Let $\ell\in \omega$ and $h' \in H$ be so that $g^{-1}g_k = g_\ell h'$
and notice that, by definition of the shift, we are attempting to show
that the following set is large.
\begin{align}
&\setof{C \in B^\I_{n, j} \suchthat \forall h \in C~ g \cdot x(g_kh)=
    m}\nonumber\\ =&\setof{C \in B^\I_{n, j} \suchthat \forall h \in
    C~ x(g^{-1}g_kh)= m}\nonumber\\ =&\setof{C \in B^\I_{n, j}
    \suchthat \forall h \in C~ x(g_\ell h' h)= m}\nonumber
\end{align}

Now define $S_j$ by the formula
\[S_j = \setof{C \in B^\I_{n, j} \suchthat \forall h \in C~ x(g_\ell h' h)= m}\]
and $T_j$ by
\[T_j = B^\I_{n, j} \setminus S_j.\]
It will suffice to show that $T_j$ is small compared to $B^\I_{n, j}$
as $j \to \infty$.  To see this, recall that $B^\I_{n, j} = A_{c_{n,
    j}} \setminus A_{c_{n, j}-1}$ and observe that each class $C$ in
$T_j$ is of one of the following three types:
\begin{enumerate}
\item $C$ is moved by $h'$ to intersect some class in $A_{c_{n,
    j}-1}$,
\item $C$ is moved by $h'$ to intersect some class outside $A_{c_{n,
    j}}$,
\item or $C$ is moved by $h'$ to hit some other class $C'$ which fails
  to specify $m$ properly.
\end{enumerate}
Thus,
\begin{align}
T_j \subseteq &\setof{C \in B^\I_{n, j} \suchthat \exists C' \in
  A_{c_{n, j}-1}~ h'C\cap C' \neq \emptyset}\nonumber\\ \cup &
\setof{C \in B^\I_{n, j} \suchthat h'C \nsubseteq \cup A_{c_{n,j}} }\nonumber\\ \cup &
\setof{C \in B^\I_{n,j} \suchthat h'C \subseteq \cup B^\I_{n, j} \wedge \ \exists C'
( h'C\cap C' \neq \emptyset \wedge 
\exists \hat h \in C' \ x(g_\ell \hat h)\neq m)}.\nonumber
\end{align}

We want to show that
\[\lim_{j \to \infty} \frac{\abs{T_j}}{\abs{B^\I_{n, j}}}=0\]
and we will do so by showing that the limit is $0$ for each of the
three sets above whose union contains $T_j$.

The first set has size at most $b(h')\abs{A_{c_{n, j}-1}}$, which is
small compared to $\abs{B^\I_{n, j}}$ as $j \to \infty$.  The second
set is small compared to $|A_{c_{n, j}}|$ by the weak amenability
hypothesis, and so also small compared to $|B^\I_{n,j}|$. 
The third set has size at
most $b(h'^{-1}) \abs{T_j'}$, where $T_j'=\setof{C' \in B^\I_{n, j}
  \suchthat \exists \hat{h} \in C'\ x(g_\ell \hat{h}) \neq m}$, which is small
compared to $\abs{B^\I_{n, j}}$ as $j \to \infty$ since $m$ witnesses $x \in
R^\I_n$.

Thus the rule sets $R^\I_n$ and $R^\II_n$ are all invariant.

Now we define the payoff for the auxiliary game. Via the bijection
$\pi \colon \omega \to G$ we have a natural bijection between $X^\omega$
and $X^G$. The auxiliary game is officially a subset of $X^\omega$, but
we view it as a subset of $X^G$ with this bijection. Thus, for a
position $n$ in the game, the move $y(n)$ is viewed
as giving the value $\tilde{x}(\pi(n))$, where $\tilde{x}\in X^G$
is the function the players are jointly building.
The payoff $\tilde{A} \subseteq X^G$ for player $\I$ in the auxiliary
game is given by:

\[\tilde A = \bigcup_{n \in \omega} \left( \bigcap_{i \leq n} R^\I_i
\setminus R^\II_n\right) \cup \left(\bigcap_{n \in \omega} \left(R^\I_n
\cap R^\II_n\right)\cap f^{-1}(A)\right)\]

\noindent
where $f$ is the following decoding function with domain $\bigcap_{n \in \omega}
\left(R^\I_n \cap R^\II_n\right)$. For $\tilde{x}\in \bigcap_{n \in \omega}
\left(R^\I_n \cap R^\II_n\right)$, define

\[f(\tilde{x})(2n)=m \Leftrightarrow \forall k \lim_{j \to \infty}
\frac{\abs{\setof{C \in B^\I_{n, j} \suchthat \forall h \in C~ \tilde{x}(g_kh)= m}}}{\abs{B^\I_{n, j}}}=1\]
\[f(\tilde{x})(2n+1)=m \Leftrightarrow \forall k
\lim_{j \to \infty}\frac{\abs{\setof{C \in B^\II_{n, j} \suchthat
\forall h \in C~ \tilde{x}(g_kh)= m}}}{\abs{B^\II_{n, j}}}=1\]

Since all the rule sets $R_n^\I$ and $R_n^\II$ are invariant, and the
function $f$ is invariant, $\tilde{A}$ is invariant.  We want to show
that whichever player has a winning strategy in the game $\tilde{A}$ has a
winning strategy for $A$.

The rule sets $R^\I_n$, $R^\II_n$ are easily $\bP^0_3$ if $X$ is finite,
and $\bS^0_4$ if $X=\omega$. This easily gives that $\tilde{A}$ is
the union of a $\bS^0_5$ set with the intersection
of a $\bP^0_5$ set and $f^{-1}(A)$. A simple computation gives that
$f$ is $\bD^0_6$-measurable. Since $\bG$ is reasonable, $\tilde{A}\in \bG$.

Suppose now $\tilde \tau$ is a winning strategy for player $\II$ in
$\tilde A$ (the case for player $\I$ is similar but slightly easier).
We will define a winning strategy $\tau$ for player $\II$ in
$A$.

We call a class $C$ {\em declared} (at position $p$)
if $C\cap \dom(p)\neq \emptyset$.
At any position $p$ in the game $\tilde{A}$ only
finitely many digits of the resulting real $\tilde{x} \in X^G$ have been
determined, and thus only finitely many classes $C$ have been
declared.   For each position $p$ and declared class $C$ relative to $p$, we have exactly one of the
following:
\begin{enumerate}
\item for some $m \in X$, for all moves $p(g)$ so far played with $g\in C$
  we have $p(g)=m$, (we say in this case that $C$ is a \emph{$m$-class})
\item or there moves $p(g)$, $p(h)$ played so far with $g,h \in C$
for which $p(g)\neq p(h)$. (we say in this case that $C$ is an \emph{invalid class}).
\end{enumerate}
For every position $p$ of the game, any class $C$ is either
undeclared, an $m$-class for some unique $m$, or an invalid class.
Over the course of the game, a class can change from an
undeclared class to an $m$-class, and will then either remain
an $m$-class or become an invalid class at some point. 
Note that invalid classes can never change.  Thus the
players' progress towards following the rules can actually be measured
at a finite position $p$.

We say a ring $B^\I_{n,j}$ or $B^\II_{n,j}$ is declared
relative to a position $p$ if all of the classes $C$ in this
ring are declared relative to $p$. 
Consider one of the sets $B^\I_{n,j}$ (or $B^\II_{n,j}$).
Suppose $p$ is a position of the auxiliary game and
$B^\I_{n,j}$ is a declared ring. 
We say (relative to the position $p$)
that $B^\I_{n,j}$ is an invalid ring if at least $1/10$ of the classes $C\in B^\I_{n,j}$
are invalid at position $p$. We say $B^\I_{n,j}$ is an $m$-ring
if at least $1/2$ of the classes $C\in B^\I_{n,j}$ are $m$-classes. 
If $B^\I_{n,j}$ is invalid at some position $p$, and $q$ is a position
which extends $p$, then $B^\I_{n,j}$ is also invalid with respect to $q$. 
If $B^\I_{n,j}$ (or $B^\II_{n-j}$) is declared but not an $m$-ring
with respect to $p$, then it is not an $m$-ring with resprct to
any extention $q$ of $p$.


Consider the first round of the game $A$. Suppose $\I$ makes first move $m_0$
in this game, and we define $\tau(m_0)$. Consider the set $P_{m_0}$ of positions
$p$ of even length in the auxiliary game $\tilde{A}$ satisfying:
\begin{enumerate}
\item
$p$ is consistent with $\tilde{\tau}$.
\item
For every $j$, every class $C\in B^\I_{0,j}$, and every $g \in C\cap \dom(p)$,
we have that $p(g)=m_0$.
\end{enumerate}
First note that we cannot have a sequence $p_0, p_1,\dots$ in $P_{m_0}$
with $p_{n}$ extending $p_{n-1}$ for all $n$, and for each $n$ there is
a $j$ such that $B^\II_{0,j}$ is invalid with respect to $p_n$ but either not invalid or not declared with
respect to $p_{n-1}$. For otherwise the limit of the $p_n$ would give a run
by $\tilde{\tau}$ for which $\I$ has satisfied the rule $R^\I_0$ but $\II$
has not satisfied $R^\II_0$, contradicting that $\tilde{\tau}$ is winning for $\II$.
Let $q_0\in P_{m_0}$ be such that there is no extension of $q_0$ in $P_{m_0}$
which a new ring $B^\II_{0,j}$ becomes invalid. So, for all sufficiently large $j$
and any $q$ extending $q_0$, the ring $B^\II_{0,j}$ is not invalid.
Likewise we cannot have a sequence $q_1 \subseteq q_2 \subseteq \cdots$
of postions extending $q_0$ such that for each $n$ there is a $j$ so that
$B^\II_{0,j}$ is declared and not an $m$-ring (for any $m$) at $q_n$, but is not declared
at $q_{n-1}$. For in this case each of these rings $B^\II_{0,j}$ would remain
not $m$-rings in the limiting run, which again violates $R^\II_0$. By extending $q_0$
we may assume that for all sufficiently large $j$ and all $q$ extending $q_0$,
$B^\II_{0,j}$, if it is declared at $q$,  is an $m$-ring for some $m$ at $q$.
Note that this ring will remain an $m$-ring for all further extension $r$ of
$q$, since it cannot change to become an invalid ring or an $m'$ ring for any $m'\neq m$. 
Finally, a similar argument shows that we cannot have a sequence $q_1 \subseteq q_2\subseteq \cdots$
extending $q_0$ such that for each $n$ there are two rings $B^\II_{0,j}$ and
$B^\II_{0,j'}$ declared at $q_n$ but not declared at $q_{n-1}$ with $B^\II_{0,j}$
an $m$-ring and $B^\II_{0,j'}$ an $m'$-ring with $m \neq m'$.
By extending $q_0$ further we may asssume that we have a declared $m_1$-ring
$B^\II_{0,j}$ with respect to $q_0$, and such that for all extexions $q$ of $q_0$
and all $j'>j$, if $B^\II_{0,j'}$ is declared at $q$ then it is also an  $m_1$-ring
We define $\tau(m_0)=m_1$.

In general, suppose $\I$ has played $m_0,m_2,\dots, m_{2k}$ in $A$.
Suppose inductively we have defined positions $q_0\subseteq q_1\subseteq \cdots  \subseteq q_{k-1}$.
Let $P_{m_0,\dots,m_{2i}}$, for $i \leq k$, be the set of positions $p$
in $\tilde{A}$ extending $q_{i-1}$ such that for all $g \in (\dom(p)\setminus
\dom(q_{i-1})) \cap B^\I_{i',j}$, for some $i'\leq i$, we have $p(g)=m_{2i'}$. 
We inductively assume that for all sufficiently large $j$ and any
$q$ extending $q_{k-1}$ in $P_{m_0,\dots,m_{2k-2}}$, and for any $i <k$, if $B^\II_{i,j}$ is declared
at $q$ then it is an $m_{2i+1}$-ring, where $m_{2i+1}=\tau(m_0,\dots,m_{2i})$. 
We now consider extensions of $q_{k-1}$ in $P_{m_0,\dots,m_{2k}}$.
By the same arguments as above, there is a $q_k\in P_{m_0,\dots,m_{2k}}$ extending $q_{k-1}$
such that for all large enough $q$ and all extensions $q$ of $q_k$,
if $B^\II_{k,j}$ is declared at $q$, then it is an $m_{2k+1}$-ring for some
fixed integer $m_{2k+1}$. We let $\tau(m_0,\dots,m_{2k})=m_{2k+1}$.
This complete the definition of the strategy $\tau$.

To see that $\tau$ is winning, note that for any run $x$ according to
$\tau$, we have a sequence of positions $q_0, q_1, \dots$ which give
us a run $\tilde x$ consistent with $\tilde \tau$ which follows all
the rules. From the definition of the $q_i$
we have that $f(\tilde{x})=x$. 
Thus since $\tilde \tau$ is winning in $\tilde A$, we
know that $\tilde x \notin f^{-1}(A)$, and so
$x \notin A$, resulting in a win for player $\II$ in the game $A$.
\end{proof}

\section{$\ged$} \label{sec:ged}

We now present a notion of $\ged$ for more general equivalence relations on Polish spaces.

\begin{definition} \label{ged}
Let $\bG$ be a pointclass and $E$ an equivalence relation on a Polish space $X$.
We say $\ged$ holds if for all continuous, onto $\pi\colon \ww \to X$ and all
$A\subseteq X$ which are $E$-invariant $\bG$ sets, we have that
$\pi^{-1}(A)$ is determined.
\end{definition}

First we note that the restriction that the coding maps $\pi$ be onto is necessary
to avoid trivialities. For by the Silver dichotomy, for every Borel equivalence
relation $E$ on $X$ with uncountably many classes, there is a continuous map $\pi
\colon \ww \to X$ such that $x \neq y$ implies $\neg \pi(x) E\, \pi(y)$. Given
$A\subseteq \ww$ in some pointclass $\bG$, let $B\subseteq X$ be the $E$-saturation of $\pi(A)$.
Then $B$ is an $E$-invariant subset of $X$ which also lies in $\bG$ for most pointclasses
(in particular for all pointclasses closed under $\exists^\ww$ or $\forall^\ww$). 
Since $\pi^{-1}(B)=A$ we see that if we allow non-surjective maps $\pi$
in Definition~\ref{ged} then $\ged$ trivially implies $\bG$-determinacy (even restricting
the maps to be continuous).

Another possible variation of Definition~\ref{ged} would be to allow Borel onto maps
$\pi \colon \ww \to X$. Although we do not see that this version trivializes the notion,
it seems more natural to require the coding maps to be as effective as possible.

A common situation is that we wish to impose a set of ``rules'' on the players in the game
$\pi^{-1}(A)$. We next show that a more general version of $\ged$ which allows for
rules imposed on the game is in fact equivalent to $\ged$ as in Definition~\ref{ged}.

\begin{definition}
Let $T\subseteq \omega^{<\omega}$ be a pruned tree (i.e., $T$ has no terminal nodes).
We say $\getd$ holds if for every continuous onto map $\pi \colon [T] \to X$ and every
$E$-invariant $\bG$ set $A\subseteq X$, we have that game $G(\pi^{-1}(A),T)$ with
payoff set $\pi^{-1}(A)$ and rule set $T$ is determined. 
\end{definition}

\begin{theorem}
For every pointclass $\bG$, for every equivalence relation $E$, we have that $\ged$ implies 
$\getd$ for every $T$.
\end{theorem}

\begin{proof}
Assume $\ged$ and let $T\subseteq \omega^{<\omega}$ be a pruned tree, and let
$\pi \colon [T] \to X$ be a continuous, onto map. Fix an $E$-invariant $\bG$ set
$A\subseteq X$. We define a continuous onto map $\pi' \colon \ww \to X$
which extends $\pi$. Let $x \in \ww \sm [T]$, and we define $\pi'(x)$.
Let $s$ be the least initial segment of $x$ with $s \notin T$. Let $s'= s\res (|s|-1)$.
We consider two cases. First suppose that $N_{s'} \cap \pi^{-1}(A)\neq \emptyset$
and $N_{s'} \cap \pi^{-1}(X\sm A)\neq \emptyset$.  Fix $x_s$, $y_s$ in $[T]$
with $\pi(x_s)\in A$ and $\pi(y_s) \notin A$. If $|s|-1$ is even (i.e.  player $\I$
was responsible for first violating the rules) then we set $\pi'(x)=\pi(y_s)$.
Likewise, if $|s|-1$ is odd, we set $\pi'(s)=\pi(x_s)$. Next suppose that
$N_{s'} \cap [T] \subseteq \pi^{-1}(A)$ or $N_{s'} \cap [T] \subseteq \pi^{-1}(X\sm A)$.
Note that the game is essentially decided at this point, so our intention is to ignore
further violations of the rules. In this case let $\pi'(x)= \pi(\ell(x))$ where
$\ell \colon \ww \to [T]$ is a fixed Lipschitz continuous retraction of $\ww$ to $[T]$.
We clearly have that $\pi'$ is continuous and extends $\pi$.

By the assumption of $\ged$, the game ${\pi'}^{-1}(A)\subseteq \ww$ is determined. 
Say without loss of generality that $\sigma'$ is a winning strategy
for $\I$ in ${\pi'}^{-1}(A)$. Let $\sigma = \ell \circ \sigma'$ (as $\ell$ is
Lipschitz, we may view $\sigma'$ as defined on sequences $s \in \omega^{<\omega}$).
We show that $\sigma$ is winning for $\I$ in $G(\pi^{-1}(A),T)$. 
Since $\sigma =\ell \circ \sigma'$, $\I$ following $\sigma$ will never 
first move off of the tree $T$. So we assume therefore $\II$ always moves in the
tree $T$. Consider a run $x$ of $G(\pi^{-1}(A),T)$ where $\I$ follows $\sigma$ and
both players move in $T$. If for every even $n$ we have that $\sigma'(x\res n)=
\sigma(x\res n)$, then $x$ is also a run of $\sigma'$ and so $x \in \pi^{-1}(A)$.
Suppose that there is a least (even) $n$ so that $\sigma'(x \res n) \neq \sigma(x\res n)$,
that is $(x\res n)\conc \sigma'(x\res n) \notin T$. Let $s'= x\res n$
and $s= s' \conc \sigma'(s')$. 
We cannot be in the first case above
(that is, both $N_{s'} \cap \pi^{-1}(A)$ and  $N_{s'} \cap \pi^{-1}(X\sm A)$
are non-empty), as otherwise $\pi'(x) = y_{s'} \notin A$, and would
be a loss of $\I$ in $\pi'^{-1}(A)$, a contradiction. In the second case we have
either $N_{s'}\cap [T] \subseteq \pi^{-1}(A)$ or $N_{s'}\cap [T] \subseteq \pi^{-1}(X\sm A)$.
We cannot have that $N_{s'}\cap [T] \subseteq \pi^{-1}(X\sm A)$ since then
$N_{s'}\subseteq \pi'^{-1}(X\sm A)$, which contradicts $\sigma'$ being winning for $\I$.
So we have $N_{s'}\cap [T] \subseteq \pi^{-1}(A)$ and so since $\sigma=\ell\circ \sigma'$,
$x \in N_{s'}\cap [T]\subseteq \pi^{-1}(A)$, and so $\I$ has won the run following $\sigma$.
\end{proof}

\section{Subshifts of finite type} \label{sec:ss}

In this section we consider $\ged$ where $E$ is the equivalence relation corresponding
to a subshift $X\subseteq 2^\Z$ of finite type. Recall this means that there is a finite set
of ``forbidden words'' $w_1,\dots,w_e \in 2^{<\omega}$ and $X= \{
x \in 2^\Z\colon \forall k \in \Z\ \forall \ell \leq e\ x\res [k, k+ |w_\ell|]\neq w_\ell \}$,
where $|w|$ denote the length of the word $w$.

\begin{theorem} \label{thm:subshift}
Let $E$ be the shift  equivalence relation on a subshift $X\subseteq 2^\Z$ of finite type,
and assume $E$ has uncountably many classes. Then for all reasonable
pointclasses $\bG$, $\ged$ implies $\bG$-determinacy.
\end{theorem}

\begin{proof}
Let $w_1,\dots,w_e$ be the forbidden words of the subshift $X$.
Fix $N \geq \max \{ |w_i|\colon 1\leq i \leq e\}$. Let $G$ be the finite directed graph, the de Bruijn
graph, corresponding to the forbidden words and $N$. That is, the vertices of $G$
are elements of $2^N$ which don't contain any forbidden words, and
$(u,v)$ is an edge in $G$ iff $v\res [0,N-1]=u\res [1,N]$.

Throughout the rest of this section, $G$ will denote this fixed de Bruijn graph (and not
a countable group).

\begin{definition}
Let $u\in G$. We say $u$ is {\em good to the right} there are uncountably many
directed paths $p=(u, u_1,u_2,\dots)$ in $G$ starting from $u$. 
Likewise we say $u$ is good to the left if there are uncountably many
$(\dots, u_2,u_1,u)$ paths starting from $u$ and moving in the reverse
direction in $G$ (i.e., $(u_{n+1},u_n)$ is an edge in $G$).
\end{definition}

If $v$ is good to the right and there is a path $u=u_0,u_1,\dots,u_n=v$
from $u$ to $v$ in the graph $G$, then $u$ is good to the right as well.
Likewise in this case, if $u$ is good to the left, then $v$ is also good
to the left. This simple observation will be used throughout.

Note that an element of $X$ can be identified with a bi-infinite path through $G$.

\begin{definition}
A {\em double loop} in a directed graph $G$ is a directed subgraph consisting
of the union of two cycles with vertex sets $C_1$ and $C_2$ such that $C_1\cap C_2 \neq \emptyset$
and $C_1\neq C_2$. See Figure~\ref{fig:dl}.
\end{definition}

\begin{figure}
\begin{tikzpicture}
  \path [draw=black,postaction={on each segment={mid arrow=black}}]
  (0,0) node {$C_1$} circle(1)
  (1,0) node[below=0.7] {$C_2$} arc (90:-180:1 and 1)
  ;
\end{tikzpicture}
\caption{a double loop}\label{fig:dl}
\end{figure}
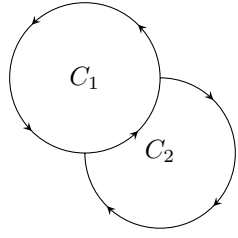

\begin{lemma} \label{edl}
If $G$ is a finite directed graph with uncountably many paths,
then $G$ contains a double loop.
\end{lemma}

\begin{proof}
Let $u_0$ be a vertex in $G$ for which there are uncountably many directed paths in $G$ starting
from $u_0$. Inductively define $(u_0,u_1,\dots,u_n)$, a directed path in $G$, such that
there are uncountably many directed paths in $G$ starting from $u_n$. We can clearly continue this
construction until we reach  a least $n$ so that $u_n \in \{ u_0,\dots,u_{n-1}\}$ (since $G$ is finite). 
Let $j\leq n-2$ be such that $u_n=u_j$, 
Let $p_0$ denote this directed path $(u_0,u_1,\dots, u_n)$. Let $\ell_0=(u_j,u_{j+1},\dots, u_{n-1})$
be the ``loop'' portion of $p_0$. 
There must be a vertex $v_0=u_k$ of $ \ell_0$
such that there is an edge $(v_0,v_1)$ in $G$ where $v_1 \neq u_{k+1}$ and such that there are uncountably many
directed paths in $G$ starting from $v_1$. For if not, then the only directed paths in $G$ starting from
$v_0$ would be, except for a countable set, those which continually follow the loop $\ell_0$.
This is a contradiction to the definition of $v_0$ as there are only countably many such paths.
If $v_1\in p_0$ then we have a double loop in the graph. If not, then we repeat the process starting at $v_1$
forming a path $p_1=(v_0,v_1,\dots, v_{m-1},v_m)$ where $v_m \in p_0\cup (v_0,\dots,v_{m-1})$.
If $v_m\in p_0$, then we have a double loop in $G$. Otherwise $v_m=v_i$ for some $i\leq m-2$, and
$\ell_1=(v_i,\dots,v_{m-1})$ gives another loop in $G$. Since $G$ is finite,
we must eventually produce a double loop in $G$. 

\begin{figure}
\begin{tikzpicture}
  \path [draw=black,postaction={on each segment={mid arrow=black}}]
  (-2,-1)  node[right=0.25] {$p_0$} arc (-180:0:0.5 and 0.25)
  (0, 0) node[below=0.7] {$\ell_0$} arc (90:-270:1 and 1)
  (1,-1)  node[right=0.25] {$p_1$}arc (-180:0:0.5 and 0.25)
  (3, 0) node[below=0.7] {$\ell_1$} arc (90:-270:1 and 1)
  (4,-1)  arc (-180:-90:0.5 and 0.25)
  (5, -1) node {$\cdots$}
  (5.5,-1.25)  node [above] {$p_n$} arc (-90:0:0.5 and 0.25)
  (7, 0) node[below=0.7] {$\ell_n$} arc (90:-270:1 and 1)
  (7, -2) arc (0:-180:2 and 0.5)
  ;
\end{tikzpicture}
\caption{}
\end{figure}

\end{proof}

Note that every vertex in a double loop is good to the right and left. 
Returning to the proof of the theorem, by our assumptions, the de Bruijn graph $G$ for the subshift has 
a double loop, and thus $G$ has a vertex which is good to the right and left.

We now define the continuous onto map $\pi\colon [T] \to X$, where $T$ will be be a pruned tree on
$\omega$ which we will be implicitly defining as we describe $\pi$. 
Let $x \in \ww$, and we describe the conditions on $x$ which give $x \in [T]$, and in this case
describe $\pi(x)\in X\subseteq 2^\Z$. First, view every digit $i$ as coding
a binary sequence $u_i$ of length $N$ (recall $N$ was maximum size of the forbidden words,
and was used to construct the de Bruijn graph $G$). Fix a fast growing sequence $0=b_0<b_1<\cdots$,
with say $\lim_i\frac{\sum_{j<i}b_j}{b_i}=0$, and with $N | (b_{i}-b_{i-1})$ for all $i$.
Let $A_i=[b_{i-1},b_{i})$ for $i \in \omega^+$.
Let also $A_{-i}=-A_i-1=\{ -a-1\colon a \in A_i\}$.

Let $n \mapsto u_n$ be an onto map from $\omega$ to $2^N$ (binary sequences of length $N$).
We first define a preliminary map $\pi'$. Given $x \in \ww$ we define $y'=\pi'(x)$
as follows. Given the first $2n$ digits of $x$, that is, $x \res 2n$, we define
$y' \res [-nN,nN)$. Assume $y'\res [-nN,nN)$ has been defined, and from $x(2n)$, $x(2n+1)$
we extend to $y'\res [-(n+1)N,(n+1)N)$. If $nN\in A_i$ for $i$ odd, then
$y'\res [nN, (n+1)N)=u_{x(2n)}$ and $y'\res [-(n+1)N, -nN)=u_{x(2n+1)}$.
If $i$ is even, then we let $y'\res [nN, (n+1)N)=u_{x(2n+1)}$ and 
$y'\res [-(n+1)N, -nN)=u_{x(2n)}$.
Let $\tilde{T}\subseteq \omega^{<\omega}$ be the tree of sequences $x \res 2n$ such that
the corresponding sequence $y' \res [-nN,nN)$ contains no forbidden words.
Note that  $[\tilde{T}]= {\pi'}^{-1}(X)$. Let $F=[\tilde{T}]$ be the closed set
corresponding to $\tilde{T}$, and let $T\subseteq \tilde{T}$ be the pruned tree with
$[T]=[\tilde{T}]=F$.

We now define $y=\pi(x)$.
If $\I's$ first move $u_{x(0)}$ is bad in both directions, we set $\pi(x)=\pi'(x)$.
If for all $n$ we have that $u_{x(n)}$ is good to the left and right,
then we also set $\pi(x)=\pi'(x)$. Otherwise, say $n_0$ is least such that $u_{x(2n_0)}$
or $u_{x(2n_0+1)}$ is not good to both the left and right.

Suppose that $u_{x(2n_0)}$ is not good in both directions, that is, $\I$ has made the
first such move. Note that $u_{x(2n_0)}$ must be good in one direction as
we are assuming that $\I$'s first move $u_{x(0)}$ is good in at least one direction.
In this case we ignore the remainder of $\I$'s moves in the game, and we alternate concatenating
$\II$'s moves on the left and right. This produces $y=\pi(x)$. If $\II$
first makes the move $u_{x(2n_0+1)}$ which is not good in both direction, we similarly ignore
the future moves of $\II$ and alternating concatenating $\I$'s move to the left and right. 
This defines $y=\pi(x)$ in all cases.

Let $A\subseteq 2^\omega$ be a $\bG$ set. We define the auxiliary payoff
set $\tA \subseteq X$ to which we will apply the hypothesis of $\ged$
for the $\pi$ constructed above. The set $\tA$ will be a shift-invariant
$\bG$ set. The construction of $\tA$ will be similar to that of
Theorem~\ref{thm:ga}. 
Say $x \in [T]\subseteq \omega^\omega$ is the play of the game $\pi^{-1}(\tA)$
where both players have followed the rules (i.e., no forbidden words appear
in $y=\pi(x)$).

First, if there  exists a subword of length $N$ in $y$ which is bad to
right and also a subword which is bad to the left, 
then $\I$ loses (this is part of the definition of $\tA$).
In the remaining cases we assume that
there is at least one direction so that all subwords of $y$ are good in that direction.

Recall the sets $A_i$ have been defined for $i\in \Z\setminus\{ 0\}$.
Let $A'_i$ be defined (for $i \in \Z\sm \{0\}$) by $A'_i=A_i$ if every subword of $y$ is good to the right,
and otherwise let $A'_i=A_{-i-1}$. 
As in the proof of Theorem~\ref{thm:ga}
we partition $\omega$ into $\omega$-many pairwise disjoint sets $B^\I_{n,j}$,
$B^\II_{n,j}$. These sets are defined exactly as before using the sets $A'_i$ for $i>0$,
except (for notational convenience) we let the $B^\I_{n,j}$ correspond
to odd $i$, and the $B^\II_{n,j}$ to even $i$.

In order to define the rule sets for the game
we will make use of the following notion.

\begin{definition}
Given $y \in X$ and a double loop $(C_0,C_1)$ in $G$, and given integers $a<b$,
we say $y\res [a,b)$ {\em traces the double loop with pattern}  $s\in 2^{<\omega}$
if $v_0=y\res[a,a+N) \in C_0 \cap C_1$, and if $v_i=y\res [a+i, a+i+N)$, then
the sequence $v_0,v_1,\dots,v_{b-a-N+1}$ of nodes in $G$ is a path in $G$ of the form
$C'_{s(0)}\conc C'_{s(1)} \conc \cdots C'_{s(|s|-1)}$ where $C'_0$, $C'_1$
are the same cycles as $C_0$, $C_1$ except we start at the vertex $v_0$.
If $|s^{-1}(i)| >|s^{-1}(1-i)|$, we call $C_i$ the {\em majority loop} and
$C_{1-i}$ the {\em minority loop}.
\end{definition}

We define the conditions $R^\I_n$, $R^\II_n$ and the decoding
function $f \colon X\to 2^\omega$ as follows.
We fix an ordering on the cycles of $G$ (the de Bruijn graph) as so write each
double loop in $G$ as $(C_0,C_1)$ where $C_0 <C_1$. This makes the representation of
each double loop unique. In the definition of $R^\I_n$ (and likewise for
$R^\II_n$) we require that  the following hold:

\begin{align} \label{RIn}
y \in R^\I_{n}  \leftrightarrow & \exists i \in \{ 0,1\}\ \forall \epsilon >0\
\exists j_0\ \forall j \geq j_0\ \exists \text{ double loop } (C_0,C_1)
\text{ such that the following hold:}
\\ & \nonumber \exists [a,b) \subseteq B^\I_{n,j}\ (b-a)> (1-\epsilon) |B^\I_{n,j}|
\text{ and } y\res[a,b) \text{ traces } (C_0,C_1) \text{ with pattern } s\in 2^{<\omega}
\\ \nonumber  & |s^{-1}(i)|\geq (1-\epsilon) \frac{2}{3} |s|
\\ &  \nonumber |s^{-1}(1-i)|\geq (1-\epsilon) \frac{1}{10} |s|
\end{align}

Likewise we define $R^\II_n$ using the $B^\II_{n,j}$ sets.  If there is some direction so that every subword of $y$ is good in that direction and $y \in \bigcap_n R^\I_n \cap \bigcap_n
R^\II_n$, then we define the decoding map $f$ at $y$ by $f(y)(2n)$ is
the witness $i \in \{ 0,1\}$ to $y \in R^\I_n$, and likewise
$f(y)(2n+1)$ is the witness to $y \in R^\II_n$.

To ensure $f$ is
well-defined, we note that if $i$ witnesses $y \in R^\I_n$, then $1-i$
does not. 
This is because for small enough $\epsilon$, for all large enough $j$ if $y$ traces
$(C_0, C_1)$ with pattern $s$ over a subinterval $I$ of length $(1-\epsilon) |B|$ of some fixed
$B=B^\I_{n, j}$ then the double loop $(C_0,C_1)$ is unique (for this $j$).
Say $C_i$ ($i\in \{ 0,1\}$) is the majority loop. The majority loop $C_i$ is traced say $M$ times, where 
$M\geq (1-\epsilon) \frac{2}{3}|s|$ many times.
The minority loop  is traced say $m$ times, where
$m\leq (\frac{1}{3}+\frac{2}{3}\epsilon)|s|+ \epsilon |B|$. The condition $m \leq M$ becomes
$\epsilon |B|\leq (\frac{1}{3}-\frac{4}{3}\epsilon)|s|$. Since $|s|\geq (1-\epsilon)\frac{|B|}{|G|}$
it suffices to have $\epsilon < (\frac{1}{3}-\frac{4}{3}\epsilon) \frac{(1-\epsilon)}{|G|}$.
Clearly this is satisfied for $\epsilon$ small enough. 
On the other hand, $m \geq (1-\epsilon)\frac{1}{10}|s|\geq (1-\epsilon)^2\frac{1}{10}\frac{|B|}{|G|}$.
Any other loop can be traced only in $B\sm I$, and so can be traced at most $\epsilon |B|$ many times.
But $(1-\epsilon)^2\frac{1}{10}\frac{|B|}{|G|}\geq \epsilon |B|$ holds for all small enough $\epsilon$.
So, there is a fixed $\epsilon$, independent of $j$, so that if $y\res B^\I_{n,j}$
satisfies Equation~\ref{RIn} for this $\epsilon$ then the double loop $(C_0,C_1)$ is well-defined
as is the integer $i \in \{0,1\}$ with $C_i$ being the majority loop.

On the invariant set of $y$ such that there is at least one direction so that all
subwords are good in that direction we define $\tA$
by:
\[
\tilde A = \bigcup_{n \in \omega} \left( \bigcap_{i \leq n} R^\I_i
\setminus R^\II_n\right) \cup \left(\bigcap_{n \in \omega} \left(R^\I_n
\cap R^\II_n\right)\cap f^{-1}(A)\right)
\]

This completes the definitions of $\pi$, and the auxiliary game $\tilde{A}$.

We next observe that the auxiliary game $\tilde{A} \subseteq 2^\Z$ is shift invariant.
Given $y \in 2^\Z$, the case split as to whether there is a direction so that all
all subwords of $y$ are good in that direction is clearly shift invariant (and the set
of such directions is also invariant). In the case where there is at least one such good direction,
whether $y \in \tilde A$ is decided by
putting down the sets $A'_i$ for $i \in Z\sm \{ 0\}$,  
using these to define the sets $B^\I_{n,j}$, $B^\II_{n,j}$, then defining the sets
$R^\I_n$, $R^\II_n$ as in  Equation~\ref{RIn} which gives the decoding function $f$
and finally asking whether $f(y)\in A$. It suffices to show that the sets $R^I_n$, $R^\II_n$
are invariant (in that $y \in R^I_n$ iff $m \cdot y \in R^I_n$ for all $m \in \Z$), as this implies
that the decoding function $f$ is also invariant. The intervals $B^\I_{n,j}(m\cdot y)$
as defined for the shift $m \cdot y$ are just the shifts $m \cdot B^\I_{n,j}(y)$ of the
corresponding sets $B^\I_{n,j}(y)$ for $y$. In particular,
$|B^\I_{n,j}(y) \cap B^\I_{n,j}(m\cdot y)| / |B^\I_{n,j}(y)|$ tends to $1$
as $j$ goes to infinity. Thus the asymptotic condition of Equation~\ref{RIn}
holds for $y$ iff it holds for $m \cdot y$.
The value of $i$ in Equation~\ref{RIn}
is therefore the same for both $y$ and $m\cdot y$. This shows that $f(y)=f(m\cdot y)$
and so $y \in \tilde A$ iff $m\cdot y \in \tilde A$.

By the assumption of $\ged$ the game $\pi^{-1}(\tilde{A})$ on $\omega$ is determined.
First consider the case where $\I$ has a winning strategy $\sigma$ in $\pi^{-1}(\tilde A)$.

\begin{claim} \label{gbd}
If $x(0),x(1),\dots, x(2n)$ is a position of the game $\pi^{-1}(\tA)$ consistent with $\sigma$
in which all of $\II$'s moves
$u_{x(2k+1)}$, $k <n$, are good in both directions, then $u_{x(2n)}$ is good in both directions.
\end{claim}

\begin{proof}
We first note that $u_{x(2n)}$ is not bad in both directions. If $n>0$ this is clear as
$u_{x(0)}$ is good in both directions and the last move is legal, and so good in (at least) the direction
pointing back to $u_{x(0)}$. If $n=0$, and $u_{x(0)}$ is bad in both directions,
then there is no direction for which every subword of the resulting $y$ is good, which is a loss for $\I$,
a contradiction. Suppose now that $u_{x(2n)}$ were bad in exactly one direction, say
bad to the right. If the definition of $\pi$ we gave full control of all future moves to $\II$
(we ignored $\I$'s moves after this point). But $\II$ can now play moves to
violate $R^\I_0$ and thus produce a loss for $\I$, a contradiction. For example, $\II$
can move (for each of the two directions) to a cycle of $G$, and simply trace this cycle
forever. 
\end{proof}

We will construct a position $p_0$ of odd length which is consistent with $\sigma$, so that
$\sigma$ is committed to a particular witness $i_0$ to its following
of $R^\I_n$, which will be our first move.  Fix $\overline{\epsilon}$ small
enough so that for all large enough $j$, if a double loop is traced in
$B^\I_{n, j}$ meeting the conditions of Equation~\ref{RIn} for this
$\overline{\epsilon}$, then that double loop is unique.

Now consider the tree of positions of odd length which
\begin{enumerate}
\item are consistent with $\sigma$,
\item in which player $\II$ has made only moves which are good in both directions,
\item there is some $j$ such that for the partial sequence $y \res [c, d)$ constructed so far, we have $[c, d)
\cap B^\I_{0, j} = B^\I_{0, j}$ and $[c, d) \cap B^\I_{0, j+1} =
\emptyset$,
\item and that $y \res [c, d)$ doesn't satisfy the existence of a double loop as in
Equation~\ref{RIn} on $B^\I_{0, j}$ with $\epsilon=\overline{\epsilon}$.
\end{enumerate}  
This tree must be wellfounded, as a branch could be used to produce a loss for $\I$
which is consistent with $\sigma$, a contradiction.  Let $p_0'$ be a terminal
position in this tree.
Next, we consider the tree of positions of odd length extending $p_0'$ which
\begin{enumerate}
\item are consistent with $\sigma$,
\item in which player $\II$ has made only moves which are good in both directions,
\item there is some $j$ so that for the partial sequence $y \res [c, d)$
constructed so far, we have $[c, d) \cap B^\I_{0, j+1} = B^\I_{0,
j+1}$ and $[c, d) \cap B^\I_{0, j+2} = \emptyset$,
\item if $y \res [c, d)$ traces $(C_0, C_1)$ in $B^\I_{0, j}$ and $(C_0', C_1')$
in $B^\I_{0, j+1}$ satisfying the conditions of Equation~\ref{RIn} with $\epsilon=\overline{\epsilon}$ and the majority loop which $y\res [c, d)$ traces in $B^\I_{0,
j}$ is $C_i$, then the majority loop $y \res [c, d)$ traces in $B^\I_{0, j+1}$ is
$C_{1-i}'$.
\end{enumerate}  
Again, this tree must be wellfounded, since a branch would result in a
$y$ which fails to satisfy $R^\I_0$.  Let $p_0$ be any terminal node
of this tree and let $j_0$ be such that if $y \res[c_0,d_0)$ is the portion of $y$ constructed at $p_0$, then
$[c_0,d_0)\cap B^\I_{0,j_0}=B^\I_{0,j_0}$ and $[c_0,d_0)\cap B^\I_{0,j_0+1}=\emptyset$. 
Notice that for any extension $q$ of $p_0$ which is
consistent with $\sigma$ and in which $\II$'s moves are good in both directions,
with $y \res [c, d)$ the portion of $y$
constructed at $q$, we will have that $y \res [c, d)$ traces a
double loop in each $B^\I_{0, j}$ for all $j\geq j_0$,
satisfying the conditions in Equation~\ref{RIn} with $\epsilon=\overline{\epsilon}$.  Furthermore,
there will be some $i_0 \in \{0, 1\}$ so that for any 
$j\geq j_0$, if $y \res [c, d)$ traces $(C_0, C_1)$ in $B^\I_{0,j}$,
then $C_{i_0}$ will be the majority loop and $C_{1-i_0}$
will be the minority loop.

Our first move in the game $A$ will be to play this $i_0$.
Next suppose our opponent plays $i_1$.  We will again consider a tree of positions consistent with $\sigma$.
We consider only positions of odd length extending $p_0$ which 
\begin{enumerate}
\item are consistent with $\sigma$,
\item in which player $\II$ has made only moves which are good in both directions,
\item in each new $B^\II_{0, j}$ player $\II$ has moved as quickly as possible to the nearest double loop $(C_0, C_1)$ in $G$, and in $B^\II_{0, j}$ declares $C_{i_1}$ the majority loop and $C_{1-i_1}$ the minority loop (say by tracing the minority loop sufficiently many times first, then the majority loop for the rest of the time) to satisfy the conditions in Equation~\ref{RIn},
\item for the partial sequence $y \res [c, d)$ constructed so far, we have for some $j$ that
$[c, d) \cap B^\I_{1, j} = B^\I_{1, j}$ and $[c, d) \cap B^\I_{1, j+1} =\emptyset$,
\item and that $y \res [c, d)$ doesn't satisfy the conditions of Equation~\ref{RIn} on $B^\I_{1, j}$ with
$\epsilon=\overline{\epsilon}$.
\end{enumerate}  
This tree must be wellfounded, since a along a branch, $\II$ would be satisfying $R^\II_0$ (with $i_1$) but player $\I$ isn't satisfying $R^\I_1$, which would be a loss consistent with $\sigma$, a contradiction.  Let $p_1'$ be a terminal position in this tree.
Next, we consider the tree of positions of odd length extending $p_1'$ which
\begin{enumerate}
\item are consistent with $\sigma$,
\item in which player $\II$ has made only moves which are good in both directions,
\item in each new $B^\II_{0, j}$ player $\II$ has moved as quickly as possible to the nearest double loop $(C_0, C_1)$ in $G$, and in $B^\II_{0, j}$ declares $C_{i_1}$ the majority loop and $C_{1-i_1}$ the minority loop to satisfy the conditions in Equation~\ref{RIn},
\item there is some $j$ so that for the partial sequence $y \res [c, d)$
constructed so far, we have $[c, d) \cap B^\I_{1, j+1} = B^\I_{1,
j+1}$ and $[c, d) \cap B^\I_{1, j+2} = \emptyset$,
\item if $y \res [c, d)$ traces $(C_0, C_1)$ in $B^\I_{1, j}$ and $(C_0', C_1')$
in $B^\I_{1, j+1}$ satisfying the conditions of Equation~\ref{RIn} with $\epsilon=\overline{\epsilon}$
and the majority loop which $y\res [c, d)$ traces in $B^\I_{1,j}$ is $C_i$,
then the majority loop $y \res [c, d)$ traces in $B^\I_{1, j+1}$ is
$C_{1-i}'$.
\end{enumerate}  
This tree also must be wellfounded, since $\I$ must commit to a
particular $i_2$ with which to satisfy $R^\I_1$, and a branch through
this tree would have $\I$ changing its answer infinitely often,
resulting in a loss consistent with $\sigma$.  Let $p_1$ be any
terminal position in this tree, and let $i_2$ be the digit which $\I$
will declare in each new $B^\I_{1, j}$ from $p_1$ onwards.

We play the move $i_2$ in $A$ and continue playing in this manner,
using each new move $i_{2n+1}$ by our opponent to satisfy in the
auxilliary game $\pi^{-1}(\tA)$ an additional rule $R^\II_n$, and then
by moving to terminal positions $p_n$ in wellfounded trees, fix digits
$i_{2n}$ which we will play in $A$.  By the construction of these
positions $p_n$, we will have that if $y$ is the resulting element of
$X$ corresponding to the sequence of moves $i_0, i_1, \cdots$ in $A$,
then we will have $f(y)(n) = i_{n}$, and that $y \in \bigcap_n (R^\I_n
\cap R^\II_n)$, so that our strategy is winning for $\I$ in $A$.

Consider now the case where $\II$ has a winning strategy $\tau$
in the game $\pi^{-1}(\tA)$. The argument is similar to the case above
where $\I$ had the winning strategy, so we just sketch the differences. 
By Lemma~\ref{edl} there is double loop in the graph $G$, and therefore there
is a word of length $N$ which is good in both directions. Have $\I$ play in $\tilde{A}$ such a word as
their first move $u_{x(0)}$. Suppose $\I$ plays $i_0$ in the game $A$.
We proceed as in the argument above having $\I$ move as quickly as possible
(only making moves which are good in both directions) to
a double loop within each $B^I_{0,j}$, and moving to encode $i_0$ within
$B^I_{0,j}$. An analogous claim to Claim~\ref{gbd} shows that as long as $\I$
plays in this manner, $\tau$'s moves are also good in both direction.
since $\I$ is satisfying $R^\I_0$ (with the digit $i_0$), a wellfoundedness
argument as before will produce a position $p_0$ consistent with $\tau$
and a digit $i_1$ so that for all runs $y$ consistent with $\tau$ in which $\I$
plays as just described we have that $y$ satisfies $R^\II_0$ with $i_1$.
Continuing in this manner, as previously, this defines a winning strategy
for $\II$ in $A$.

\section{Conclusion} \label{sec:conclusion}

We have introduced the general notion of $\ged$ for arbitrary
equivalence relations $E$ on Polish spaces and pointclasses $\bG$. Since the definition
involves the use of continuous coding maps from $\ww$ onto $X$ (which seems necessary
to have a reasonable definition), it is not immediately clear to what extent the
structure of the particular Polish space $X$ plays a role. Since all uncountable Polish spaces are Borel
isomorphic, it is reasonable to ask the following:

\begin{question} \label{ques:fq}
Suppose $X$, $Y$ are uncountable Polish spaces, $\varphi\colon X\to Y$ is a Borel
isomorphism, and $E_X$ is a Borel equivalence relation on $X$. Let $E_Y$ be the corresponding
Borel equivalence relation on $Y$, that is $y_1 E_Y\, y_2$ iff $\varphi^{-1}(y_1) E_X\, 
\varphi^{-1}(y_2)$. Then for any pointclass $\bG$ closed under substitution by Borel
functions, countable unions and countable intersections, is it the case that 
$(\bG, E_X)$-determinacy is equivalent to $(\bG,E_Y)$-determinacy?
\end{question}

If $E,F$ are equivalence relation on $X$ with $E\subseteq F$, then
we immediately have that $(\bG,E)$-determinacy implies $(\bG,F)$-determinacy
(because $F$-invariant sets are also $E$-invariant). Since the shift equivalence
relation $E_\Z$ on $2^\Z$ is a subset of the Turing equivalence relation on $2^\Z$,
we have that $(\bG, E_Z)$-determinacy implies $\bG$ Turing-determinacy for any $\bG$.
We have shown that $(\bG,E_\Z)$-determinacy implies full $\bG$-determinacy
for any reasonable $\bG$. We recall that Harrington showed that
$\bS^1_1$ Turing-determinacy implies $\bS^1_1$-determinacy and Woodin
showed that in $L(\R)$, Turing-determinacy is equivalent to full determinacy.
Thus, extending our results to more general equivalence relations is
expected to be a difficult problem. Nevertheless, in Theorem~\ref{thm:subshift}
we extended the result to include subshifts of $2^\Z$ of finite type.

We recall that the Feldman-Moore theorem states that every countable Borel equivalence
relation $E$ on a Polish space is generated by the Borel action
of a countable group $G$ (one can also choose the Polish topology
to make the action continuous). We also recall the result \cite{djk}
that every equivalence relation $E$ generated by the action of a countable group $G$
Borel (equivariantly) embeds into the shift action of $G\times \Z$ on $2^{G\times \Z}$.
Theorem~\ref{thm:ga} applies to shift actions of arbitrary countable groups, so
the problem of passing to general (not necessarily closed) subshifts embodies the general question
of whether $\ged$ implies $\bG$-determinacy. In particular, we can ask:

\begin{question} \label{ques:sq}
For which subshifts (closed, or more generally Borel,
invariant subsets $X$ of $2^\Z$ with the shift map)
of $2^\Z$ do we have that $\ged$ implies $\bG$-determinacy, where
$E$ is the shift equivalence relation restricted to $X$.
\end{question}

If $E$ is generated by the continuous action of a countable group $G$
on a compact $0$-dimensional space $X$, then \cite{djk} shows that
$(X,E)$ equivariantly and continuously embeds into a subshift of $2^{\Z\times G}$.
Thus, given a positive answer to Question~\ref{ques:fq}, the question of whether
$\ged$ implies $\bG$-determinacy reduces to considering the question
for subshifts of $2^G$, for countable groups $G$.

Aside from the observation above on subequivalence relations, it is not clear how the
notion of $(\bG,E)$-determinacy interacts with other aspects of the theory of Borel equivalence relations.
So we ask:

\begin{question} \label{ques:tq}
How does the notion of $\ged$ interact with the notions of Borel reducibility
of equivalence relations, products of equivalence relations, increasing unions
of equivalence relations, etc.?
\end{question}

\end{proof}

\bibliographystyle{amsplain}

\bibliography{equiv}

\providecommand{\bysame}{\leavevmode\hbox to3em{\hrulefill}\thinspace}
\providecommand{\MR}{\relax\ifhmode\unskip\space\fi MR }
\providecommand{\MRhref}[2]{%
  \href{http://www.ams.org/mathscinet-getitem?mr=#1}{#2}
}
\providecommand{\href}[2]{#2}
\begin{thebibliography}{1}

\bibitem{djk}
Randall Dougherty, Stephen Jackson, and Alexander~S. Kechris, \emph{The
  structure of hyperfinite borel equivalence relations}, Transactions of the
  American Mathematical Society \textbf{341} (1994), no.~1, 193--225.

\bibitem{Harrington1978}
Leo Harrington, \emph{Analytic determinacy and {$0^\#$}}, Journal of Symbolic
  Logic \textbf{43} (1978), no.~4, 685--693.

\bibitem{KechrisBook}
Alexander~S. Kechris, \emph{Classical descriptive set theory}, Graduate Texts
  in Mathematics, vol. 156, Springer-Verlag, New York, 1995.

\bibitem{KechrisMarks}
Alexander~S. Kechris and Andrew Marks, \emph{Descriptive graph combinatorics},
  to appear.

\bibitem{Marks2016}
Andrew Marks, \emph{A determinacy approach to {B}orel combinatorics}, Journal
  of the American Mathematical Society \textbf{29} (2016), no.~2, 579--600.

\bibitem{Martin1975}
Donald~A. Martin, \emph{Borel determinacy}, Annals of Mathematics \textbf{102}
  (1975), no.~2, 363--371.

\bibitem{Martin1985}
\bysame, \emph{A purely inductive proof of {B}orel determinacy}, Recursion
  theory (Ithaca, N.Y., 1982), Proc. Sympos. Pure Math., vol.~42, Amer. Math.
  Soc., Providence, RI, 1985, pp.~303--308.

\bibitem{MoschovakisBook}
Yiannis~N. Moschovakis, \emph{Descriptive set theory}, Mathematical Surveys and
  Monographs, American Mathematical Society; 2nd edition, 2009.

\end{thebibliography}

\end{document}